\documentclass{amsart}

\usepackage[T1]{fontenc}
\usepackage{amsmath,amssymb,amsthm}
\usepackage{mathtools}
\usepackage{cite}
\usepackage[hidelinks]{hyperref}

\newtheorem{theorem}{Theorem}[section]
\newtheorem{proposition}[theorem]{Proposition}
\newtheorem{lemma}[theorem]{Lemma}
\newtheorem{corollary}[theorem]{Corollary}

\theoremstyle{definition}
\newtheorem{definition}[theorem]{Definition}

\theoremstyle{remark}
\newtheorem{remark}[theorem]{Remark}

\numberwithin{equation}{section}

\DeclareMathOperator{\Hom}{Hom}

\DeclareMathOperator{\del}{del}
\DeclareMathOperator{\sddel}{sub\text{-}ddel}
\DeclareMathOperator{\rad}{rad}

\newcommand{\op}{\mathrm{op}}
\newcommand{\Lam}{\Lambda}
\newcommand{\mat}[1]{\begin{pmatrix}#1\end{pmatrix}}
\newcommand{\smat}[1]{\begin{psmallmatrix}#1\end{psmallmatrix}}

\renewcommand{\le}{\leqslant}
\renewcommand{\ge}{\geqslant}

\begin{document}
\title[Derived equivalences and delooping levels]
{Derived equivalences and delooping levels}

\author[Chen]{Liang Chen}

\address{School of Mathematical Sciences\\
	Capital Normal University\\
	Beijing 100048\\
	China}
	
\email{2210501003@cnu.edu.cn}

\renewcommand{\subjclassname}{Mathematics Subject Classification (2020)}


\date{}

\begin{abstract}
We construct a finite-dimensional algebra derived equivalent to the example of Kershaw--Rickard. 
For the Kershaw--Rickard example the delooping level and the sub-derived delooping level are both infinite, while for our algebra both invariants are $0$. 
Thus the finiteness of these invariants is not preserved under derived equivalences.
\end{abstract}

\maketitle

\section{Introduction}
Motivated by attempts to resolve the finitistic dimension conjecture \cite{Bass60},
G\'elinas introduced in~\cite{Gelinas22Adv} the \emph{delooping level} $\del$,
a syzygy-theoretic invariant that bounds the big finitistic dimension of the opposite algebra.
Guo and Igusa~\cite{GuoIgusa25} later refined this notion by introducing, among other refinements,
the \emph{sub-derived delooping level} $\sddel$, which provides sharper homological bounds.

Derived equivalences often preserve homological finiteness properties---for instance, finiteness of the global dimension and of the little and big finitistic dimensions; see, e.g., \cite{Rickard89,Happel88,Kato02,Keller92,KellerKrause20,PanXi09,ChenXi17}.
Xi--Zhang therefore conjectured that finiteness of $\del$ is invariant under derived equivalences among Noetherian rings \cite[Conjecture in Section~5.2]{XiZhang26}.
In this note, we give a counterexample: the conjecture fails already for finite-dimensional algebras (and hence for Noetherian rings).
The analogous finiteness statement fails for $\sddel$ as well. Our main result is the following.

\begin{theorem}\label{thm:main}
There exist derived equivalent finite-dimensional algebras $B$ and $C$ over a field $k$ such that
\[
\del(B)=\sddel(B)=\infty
\qquad\text{and}\qquad
\del(C)=\sddel(C)=0.
\]
\end{theorem}

The proof of Theorem~\ref{thm:main} is given in the next section.

\section{Proof of Theorem~\ref{thm:main}}\label{sec:proof}
Throughout, $k$ denotes a field.
By an \emph{algebra} we mean a finite-dimensional $k$-algebra, and by a \emph{module} we mean a finitely generated
\emph{left} module unless explicitly stated otherwise.
For an algebra $A$, let $A\text{-}\mathrm{mod}$ be the category of finitely generated left $A$-modules and
$A\text{-}\mathrm{Mod}$ the category of all left $A$-modules.
We write $D^b(A):=D^b(A\text{-}\mathrm{Mod})$ for the bounded derived category. Two algebras $A$ and $A'$ are \emph{derived equivalent} if there exists a triangulated equivalence
$D^b(A)\simeq D^b(A')$; see \cite[Theorem~6.4]{Rickard89}.

Let $A$ be an algebra and write $A\text{-}\underline{\mathrm{mod}}$ for the
stable module category, that is, the quotient of $A\text{-}\mathrm{mod}$ by the
ideal of morphisms factoring through projective modules.
For $X,Y\in A\text{-}\underline{\mathrm{mod}}$, $X$ is a direct summand of $Y$
in $A\text{-}\underline{\mathrm{mod}}$ if and only if there exists a projective
$A$-module $P$ such that $X$ is a direct summand of $Y\oplus P$ in
$A\text{-}\mathrm{mod}$.
For $M\in A\text{-}\mathrm{mod}$, let $\Omega_A(M)$ denote the kernel of a
projective cover $P\twoheadrightarrow M$. Set $\Omega_A^0(M):=M$ and define
$\Omega_A^{n+1}(M):=\Omega_A(\Omega_A^n(M))$ for $n\ge0$.
The modules $\Omega_A^n(M)$ are well defined up to isomorphism.

We now recall the definitions of the delooping level and the sub-derived delooping level.

\begin{definition}[\textnormal{\cite[Definition~1.2]{Gelinas22Adv}}]\label{def:del}
Let $M \in A\text{-}\mathrm{mod}$. The \emph{delooping level} of $M$ is defined by
\[
\del_A(M):=\inf\Bigl\{\,n\ge 0 \ \Big|\ 
\begin{aligned}[t]
	&\Omega_A^n(M)\text{ is a direct summand of }\Omega_A^{n+1}(N)\\
	&\text{in }A\text{-}\underline{\mathrm{mod}},\ \text{for some }N\in A\text{-}\mathrm{mod}
\end{aligned}
\Bigr\}.
\]
If no such $n$ exists, we set $\del_A(M):=\infty$. The delooping level of $A$ is
\[
\del(A):=\sup\{\del_A(S)\mid S\text{ is a simple }A\text{-module}\}.
\]
\end{definition}

\begin{definition}[\textnormal{\cite[Remark~2.16]{GuoIgusa25}}]\label{def:sddel}
Let $M \in A\text{-}\mathrm{mod}$. 
The \emph{sub-derived delooping level} of $M$ is
\[
\sddel_A(M):=\inf\{\del_A(N)\mid M\hookrightarrow N \text{ in } A\text{-}\mathrm{mod}\}.
\]
If $M$ does not embed into any module of finite delooping level, this infimum is $\infty$.
The sub-derived delooping level of $A$ is
\[
\sddel(A):=\sup\{\sddel_A(S)\mid S \text{ is a simple } A\text{-module}\}.
\]
\end{definition}

\begin{remark}\label{rem:sddel-le-del}
Since $M\hookrightarrow M$, one has $\sddel_A(M)\le \del_A(M)$ for all $M$.
\end{remark}

We now recall the Kershaw--Rickard example.
Fix a field $k$ containing an element $q\in k^{\times}$ of infinite multiplicative order, that is,
$q^n\neq 1$ for all $n\ge 1$ (for instance, $k=\mathbb{Q}$ and $q=2$).
Ringel--Zhang \cite[\S6]{RZ20b} define the following local algebra and modules.
Let
\[
\Lam:=k\langle x,y,z\rangle
\big/
\bigl(
 x^2,\ y^2,\ z^2,\ yz,\ xy+qyx,\ xz-zx,\ zy-zx
\bigr).
\]
For $\alpha\in k$, let $M(\alpha)$ be the $k$-vector space with basis $v, v', v''$ and left $\Lam$-action
\[
 xv=\alpha v',\qquad yv=v',\qquad zv=v'',
\]
and $xv'=yv'=zv'=xv''=yv''=zv''=0$.
We write $M(q)$ for the special module at $\alpha=q$.
Set
\[
B:=\mat{k & 0\\[2pt] M(q) & \Lam}.
\]
Up to opposite rings, this is the Kershaw--Rickard example~\cite{KR24}.

\begin{lemma}[\textnormal{{\cite[Corollary~4.2]{KR24}, \cite[Example~3.8]{GuoIgusa25}}}]
	\label{lem:B-infinite}
	$\del(B)=\sddel(B)=\infty$.
\end{lemma}

Define
\[
C:=\mat{\Lam & 0\\[2pt] D(M(q)) & k}.
\]
Here $D(-):=\Hom_k(-,k)$.

We first show that $B$ and $C$ are derived equivalent.
To this end we use the following two lemmas.

\begin{lemma}\label{lem:swap-triangular}
	Let $R,S$ be rings and ${}_S N_R$ a bimodule.
	Then there is a ring isomorphism
	\[
	\mat{S & N\\ 0 & R}\;\cong\;\mat{R & 0\\ N & S}.
	\]
\end{lemma}
\begin{proof}
	Set
	\[
	A:=\mat{S & N\\ 0 & R},
	\quad
	A':=\mat{R & 0\\ N & S},
	\quad
	e_1=\mat{1&0\\0&0},
	\quad
	e_2=\mat{0&0\\0&1}\in A.
	\]
	Since ${}_A(Ae_1\oplus Ae_2)\cong {}_A(Ae_2\oplus Ae_1)$, we obtain
	\[
	A\cong \operatorname{End}_A(Ae_1\oplus Ae_2)^{\op}
	\cong
	\operatorname{End}_A(Ae_2\oplus Ae_1)^{\op}
	\cong A'. \qedhere
	\]
\end{proof}

\begin{lemma}[\textnormal{Ladkani \cite[Remark~4.13]{Ladkani11}}]\label{lem:Ladkani}
Let $R$ be a ring and $S$ a division ring. Let ${}_S N_R$ be a bimodule, finite-dimensional as a left $S$-vector space.
Let $D(-):=\Hom_S(-,S)$. Then the upper triangular matrix rings
\(\mat{S & N\\ 0 & R}\) and \(\mat{R & D(N)\\ 0 & S}\)
are derived equivalent.
\end{lemma}

\begin{corollary}\label{cor:derived-equivalent}
The algebras $B$ and $C$ are derived equivalent.
\end{corollary}
\begin{proof}
Apply Lemma~\ref{lem:Ladkani} with $S=k$, $R=\Lam$ and $N=D(M(q))$ (viewed as a $k$--$\Lam$-bimodule).
This yields a derived equivalence between the upper triangular rings
\[
T_1:=\mat{k & D(M(q))\\ 0 & \Lam}
\qquad\text{and}\qquad
T_2:=\mat{\Lam & D(D(M(q)))\\ 0 & k}.
\]
Since $M(q)$ is finite-dimensional over $k$, we have $D(D(M(q)))\cong M(q)$, and hence
\(T_2\cong \mat{\Lam & M(q)\\ 0 & k}\).
By Lemma~\ref{lem:swap-triangular}, we have \(T_1\cong C\) and \(T_2\cong B\).
Hence $B$ and $C$ are derived equivalent.
\end{proof}

We next compute the delooping level and the sub-derived delooping level of $C$.
Let \(e_1=\smat{1&0\\0&0}\) and \(e_2=\smat{0&0\\0&1}\) in $C$, and set
$P_1:=Ce_1$, $P_2:=Ce_2$.
Write $S_1$ (resp.\ $S_2$) for the simple top of $P_1$ (resp.\ $P_2$).
Note that $P_2=Ce_2$ is simple and projective, hence $S_2\cong P_2$.
Thus $\del_C(S_2)=0$ by Definition~\ref{def:del}.
To compute $\del_C(S_1)$, we first note two elementary facts.

\begin{lemma}\label{lem:yx-simple-ideal}
The left ideal $\Lam(yx)$ is 1-dimensional over $k$.
\end{lemma}
\begin{proof}
Using the defining relations of $\Lambda$ we compute
\begin{align*}
x(yx) &= (xy)x = (-q\,yx)x = -q\,y(x^2)=0,\\
y(yx) &= y^2x=0,\\
z(yx) &= (zy)x = (zx)x = z(x^2)=0.
\end{align*}
Thus $x,y,z$ annihilate $yx$ on the left. Moreover, the defining relations of $\Lambda$ show that $1,x,y,z,yx,zx$ form a $k$-basis for $\Lam$.
Since $x(yx)=0$, we also have $(yx)(yx)=y(x(yx))=0$ and $(zx)(yx)=z(x(yx))=0$.
Therefore $\Lam(yx)=k\cdot (yx)$ and hence $\dim_{k}\Lam(yx)=1$. 
\end{proof}

\begin{lemma}\label{lem:yx-annihilates-DM}
The element $yx\in\Lam$ annihilates $D(M(q))$ from the right: for all $f\in D(M(q))$, one has $f\cdot(yx)=0$.
\end{lemma}
\begin{proof}
The right $\Lam$-action on $D(M(q))=\Hom_{k}(M(q),k)$ is given by $(f\cdot a)(m)=f(am)$.
It suffices to show $(yx)M(q)=0$. Using the defining action of $M(q)$ we have
\[
\begin{aligned}
(yx)v &= y(xv)=y(qv')=0,\\
(yx)v' &= y(xv')=0,\\
(yx)v'' &= y(xv'')=0.
\end{aligned}
\]
Since $v,v',v''$ form a $k$-basis of $M(q)$, we have $(yx)M(q)=0$.
Hence $f\cdot(yx)=0$ for all $f\in D(M(q))$.
\end{proof}

\begin{proposition}\label{prop:S1-syzygy}
$\del_C(S_1)=0$.
\end{proposition}
\begin{proof}
Let
\(U:=\{\smat{\lambda(yx)&0\\0&0}\mid \lambda\in\Lam\}\subseteq P_1.\)
To check that $U$ is a $C$-submodule of $P_1$, we take $c=\smat{a&0\\ f&s}\in C$ and $u=\smat{\lambda(yx)&0\\0&0}\in U$.
Then
\[
cu=\smat{a\lambda(yx)&0\\ f\cdot\lambda(yx)&0}=\smat{a\lambda(yx)&0\\0&0}\in U,
\]
since $f\cdot\lambda(yx)=(f\cdot\lambda)\cdot(yx)=0$ by Lemma~\ref{lem:yx-annihilates-DM}.

As a $k$-vector space, $U\cong \Lam(yx)$.
By Lemma~\ref{lem:yx-simple-ideal}, $\Lam(yx)$ is 1-dimensional,
so $U$ is a simple $C$-module.
Since $e_2U=0$ but $e_2S_2=S_2\neq0$, it follows that $U\cong S_1$.

Since $P_1$ is indecomposable projective, it is local. 
As $U\subsetneq P_1$, we have $U\subseteq \rad(P_1)$.
Hence $P_1\twoheadrightarrow P_1/U$ is a projective cover, and the short exact sequence
$0\to U\to P_1\to P_1/U\to 0$ yields $U\cong \Omega_C(P_1/U)$.
Thus $S_1\cong U$ is a syzygy, and $\del_C(S_1)=0$ by Definition~\ref{def:del}.
\end{proof}

\begin{corollary}\label{cor:C-zero}
 \(\del(C)=\sddel(C)=0.\)
\end{corollary}
\begin{proof}
By Proposition~\ref{prop:S1-syzygy}, we have $\del_C(S_1)=0$.
Since $S_2$ is projective, $\del_C(S_2)=0$ by Definition~\ref{def:del}.
Hence $\del(C)=0$.
Moreover, $\sddel(C)\le \del(C)=0$ by Remark~\ref{rem:sddel-le-del}, and thus $\sddel(C)=0$.
\end{proof}

\begin{proof}[Proof of Theorem~\ref{thm:main}]
Let $B$ and $C$ be the algebras defined above.
Lemma~\ref{lem:B-infinite} gives $\del(B)=\sddel(B)=\infty$.
By Corollary~\ref{cor:derived-equivalent}, the algebras $B$ and $C$ are derived equivalent.
On the other hand, Corollary~\ref{cor:C-zero} yields $\del(C)=\sddel(C)=0$.
This proves Theorem~\ref{thm:main}.
\end{proof}

\paragraph*{Acknowledgements.} The author is grateful to Professor Changchang Xi for helpful discussions.


\end{document}